\documentclass[11pt,english]{amsart}
\usepackage[a4paper]{geometry}
\usepackage{amssymb}
\usepackage{amsmath}
\usepackage{fancyhdr}
\usepackage{color}
\usepackage{enumerate}
\newtheorem{theorem}{Theorem}[section] 
\newtheorem{lemma}[theorem]{Lemma}     
\newtheorem{corollary}[theorem]{Corollary}

\newtheorem{example}[theorem]{Example}
\newtheorem{definition}[theorem]{Definition}

\newtheorem{remarks}[theorem]{Remarks}

\newtheorem*{theorem*}{Theorem}

\usepackage[bookmarksopen, naturalnames]{hyperref}

\makeatother

\begin{document}

\title[Universal Taylor series with respect to a prescribed subsequence]{Universal Taylor series with respect to a prescribed subsequence}
\author{A. Mouze}
\thanks{The author was partly supported by the grant ANR-17-CE40-0021 of the French National Research Agency ANR (project Front)}
\address{Augustin Mouze, \'Ecole Centrale de Lille, CNRS, UMR 8524 - Laboratoire Paul Painlev\'e  59000 Lille, France}
\email{augustin.mouze@univ-lille.fr}


\begin{abstract} For a holomorphic function $f$ in the open unit disc $\mathbb{D}$ 
and $\zeta\in\mathbb{D}$, $S_n(f,\zeta)$ denotes 
the $n$-th partial sum of the Taylor development of $f$ at $\zeta$. Given an increasing 
sequence of positive integers $\mu=(\mu_n)$, we consider the classes $\mathcal{U}(\mathbb{D},\zeta)$ (resp. 
$\mathcal{U}^{(\mu)}(\mathbb{D},\zeta)$) of such functions $f$ 
such that the partial sums $\{S_n(f,\zeta):n=1,2,\dots\}$ (resp. $\{S_{\mu_n}(f,\zeta):n=1,2,\dots\}$) 
approximate all polynomials uniformly on the compact sets $K\subset\{z\in\mathbb{C}:\vert z\vert\geq 1\}$ with 
connected complement. 
We show that these two classes of universal Taylor series coincide if and only if $\limsup_n\left(\frac{\mu_{n+1}}{\mu_n}\right)<+\infty$. In the same spirit, we prove that, for $\zeta\ne 0,$ we have the equality $\mathcal{U}^{(\mu)}(\mathbb{D},\zeta)=
\mathcal{U}^{(\mu)}(\mathbb{D},0)$ if and only if $\limsup_n\left(\frac{\mu_{n+1}}{\mu_n}\right)<+\infty$. 
Finally we deal with the case of real universal Taylor series.
\end{abstract}

\maketitle

\section{Introduction} As usual $\mathbb{N}$, $\mathbb{Q}$ denote the sets of positive integers and rational numbers respectively. 
Let $\mathbb{D}:=\{z\in\mathbb{C}:\vert z\vert<1\}$ be the open unit disc of the complex plane. Throughout the paper, 
$H(\mathbb{D})$ denotes the vector space of all holomorphic functions on $\mathbb{D}$ endowed with the topology of uniform convergence 
on all compact subsets of $\mathbb{D}$. Also for a compact set $K$ of $\mathbb{C}$ we denote by $A(K)$ the set of all 
functions which are holomorphic in the interior $K^{o}$ of $K$ and continuous on $K$. As usual, for a holomorphic function 
$f$ in the unit disc and $\zeta\in\mathbb{D}$, $S_n(f)$ or $S_n(f,\zeta)$ stands for the $n$-th partial sum of the Taylor development 
of $f$ with center at $0$ or at $\zeta$ respectively. In 1996, Nestoridis proved the following result \cite{NestorF}.

\begin{theorem}\label{thm_NF}\cite{NestorF} There exist Taylor series $f=\sum_{n\geq 0}a_nz^n$ 
such that, for every compact set $K\subset\{z\in\mathbb{C};\vert z\vert\geq 1\}$ 
with connected complement and for every function $h\in A(K)$ there exists 
a subsequence $(\lambda_n)\subset\mathbb{N}$ such that 
$S_{\lambda_{n}}\left(f\right)$ converges to $h$, as $n\rightarrow +\infty$, uniformly on $K$.will

\end{theorem} 

In the sequel, such Taylor series will be called \emph{universal Taylor series} and we will denote by $\mathcal{U}(\mathbb{D},0)$ the set of such 
universal Taylor series. 
In the same spirit, for $\zeta\in\mathbb{D}$, we can replace $S_{\lambda_n}(f)$ by $S_{\lambda_n}(f,\zeta)$ in the previous theorem to 
obtain the class $\mathcal{U}(\mathbb{D},\zeta)$ of universal Taylor series with center $\zeta$. 
The sets $\mathcal{U}(\mathbb{D},0)$ or $\mathcal{U}(\mathbb{D},\zeta)$ enjoy very strong properties. 
For example, these sets are $G_{\delta}$ dense subsets of $H(\mathbb{D})$ and contain, apart 
from $0$, a dense vector subspace of $H(\mathbb{D})$ \cite{bgnp,NestorF}. This last property means that 
the sets $\mathcal{U}(\mathbb{D},0)$ and $\mathcal{U}(\mathbb{D},\zeta)$ are algebraically generic. A very nice theorem asserts that 
for all $\zeta\in\mathbb{D}$, $\mathcal{U}(\mathbb{D},0)=\mathcal{U}(\mathbb{D},\zeta)$. We refer the reader to \cite{GLM} and \cite{MNadv}. A crucial tool for the proof 
is a result initiated by Gehlen, Luh and M\"uller \cite{GLM} which asserts that every universal Taylor series 
actually possesses Ostrowski-gaps, in the sense of the following definition. 
 
\begin{definition}\label{ogapcomplex}{\rm Let $\zeta\in\mathbb{C}$. Let $\sum_{j=0}^{+\infty}a_j(z-\zeta)^j$ be a complex power series with radius of 
convergence $r\in (0,+\infty).$ We say that it has {\it Ostrowski-gaps} $(p_m,q_m)$ if 
$(p_m)$ and $(q_m)$ are sequences of natural numbers with 
\begin{enumerate}
\item $p_1<q_1\leq p_2<q_2\leq\dots$ and $\lim_{m\rightarrow +\infty}\frac{q_m}{p_m}=+\infty,$
\item for $I=\cup_{m=1}^{\infty}\{p_m+1,\dots,q_m\},$ we have $\lim_{j\in I}\vert a_j\vert^{1/j}=0.$
\end{enumerate}
}
\end{definition}
The fact that every universal Taylor series possesses Ostrowski-gaps is at the core of many beautiful results (see for instance 
\cite{Bay1, Char, GLM, Katso, MNadv, MouMun}). Thus the combination of the existence of Ostrowski-gaps with a result of Luh 
(see Theorem 1 of \cite{Luh1986} and Lemma \ref{lemme_npf} below) allows 
to obtain the aforementioned equality $\mathcal{U}(\mathbb{D},0)=\mathcal{U}(\mathbb{D},\zeta)$. Notice that 
we give a new proof of \cite[Theorem 1]{Luh1986} in Section \ref{preliminary_section}. 
Next, in order to prove the algebraic genericity of the class of universal Taylor series, the following subclass of universal series was introduced. 

\begin{definition} {\rm Let $\zeta\in\mathbb{D}$. Let $\mu=(\mu_n)$ be an increasing sequence of positive integers with 
$\mu_n\rightarrow +\infty$ as $n$ tends to infinity. 
A holomorphic function $f\in H(\mathbb{D})$ belongs to 
the class $\mathcal{U}^{(\mu)}(\mathbb{D},\zeta)$ if for every compact set $K\subset\{z\in\mathbb{C};\vert z\vert\geq 1\}$ 
with connected complement and for every function $h\in A(K)$ there exists 
a subsequence $(\lambda_n)\subset\mathbb{N}$ such that 
$S_{\mu_{\lambda_{n}}}(f,\zeta)$ converges to $h$, as $n\rightarrow +\infty$, uniformly on $K$.}
\end{definition} 

Obviously we have $\mathcal{U}^{(\mu)}(\mathbb{D},0)\subset \mathcal{U}(\mathbb{D},0)$ 
(or $\mathcal{U}^{(\mu)}(\mathbb{D},\zeta)\subset \mathcal{U}(\mathbb{D},\zeta)$). 
In \cite{MouMun}, as a consequence of a more general result relating to the weighted densities of subsequences along which 
the partial sums of universal Taylor series realize the universal approximation, 
the authors exhibit non-trivial subsequences $\mu$ of $\mathbb{N}$ such that  
$\mathcal{U}^{(\mu)}(\mathbb{D},0)= \mathcal{U}(\mathbb{D},0)$. For example, the sequences $(\mu_n)=(n^2),$ $(\mu_n)=(2^n)$ or the sequence of prime 
numbers satisfy this property. In \cite[Section 5]{MouMun}, it is asked to characterize the subsequences $\mu$ for which the equality 
$\mathcal{U}^{(\mu)}(\mathbb{D},0)=\mathcal{U}(\mathbb{D},0)$ holds. In this paper, we are going to answer 
this question by establishing the following result. 

\begin{theorem}\label{main_thm} 
Let $\zeta\in\mathbb{D}$. Let $\mu=(\mu_n)$ be a strictly increasing sequence of positive integers. The following assertions 
are equivalent:
\begin{enumerate}[(i)] 
\item\label{1ass1} $\mathcal{U}(\mathbb{D},\zeta)=\mathcal{U}^{(\mu)}(\mathbb{D},\zeta)$
\item\label{2ass1} $\displaystyle\limsup_{n\rightarrow +\infty}\left(\frac{\mu_{n+1}}{\mu_n}\right)<+\infty.$
\end{enumerate}
\end{theorem}

To prove the implication $(\ref{2ass1})\Rightarrow (\ref{1ass1})$ we use in an essential way the fact that all universal Taylor series possess 
Ostrowski-gaps. To obtain the converse implication, we employ a constructive method based on a Bernstein-Walsh type theorem 
given by Costakis and Tsirivas (see Theorem 
\ref{thmCT} below), when they studied the phenomenon of disjoint universal Taylor series \cite{CT}. As consequence of  
Theorem \ref{main_thm} we show the independence of the class $\mathcal{U}^{(\mu)}(\mathbb{D},\zeta)$ with the center of expansion $\zeta$ provided that 
$\displaystyle\limsup_{n\rightarrow +\infty}\left(\frac{\mu_{n+1}}{\mu_n}\right)<+\infty.$ This phenomenon was already noticed in specific cases 
in \cite{Vlachou1, Vlachou2}. Furthermore, using a constructive method similar to that of the proof of Theorem \ref{main_thm} and 
the ideas of the new proof of \cite[Theorem 1]{Luh1986}, we obtain the following characterization. 

\begin{theorem}\label{main_thm_inde} 
Let $\zeta\in\mathbb{D}$, $\zeta\ne0$. Let $\mu=(\mu_n)$ be a strictly increasing sequence of positive integers. The following assertions 
are equivalent:
\begin{enumerate}[(i)] 
\item\label{1ass1} $\mathcal{U}^{(\mu)}(\mathbb{D},\zeta)=\mathcal{U}^{(\mu)}(\mathbb{D},0)$
\item\label{2ass1} $\displaystyle\limsup_{n\rightarrow +\infty}\left(\frac{\mu_{n+1}}{\mu_n}\right)<+\infty.$
\end{enumerate}
\end{theorem}

We refer the reader to Corollary \ref{main_thm3} and Theorem \ref{rec_thm}. 
Finally in the last section we deal with the case of real universal Taylor series. 
  
\section{Universal Taylor series versus universal Taylor series with respect to a prescribed subsequence}

\subsection{Preliminary results}\label{preliminary_section} 
In this subsection we state some results that we will use for the proof of the main theorem. On one hand we are interested 
in the fact that all universal Taylor series possess Ostrowski-gaps. Actually a slightly more precise result holds (se \cite[Theorem 9.1]{MNadv}). 

\begin{lemma}\label{lemme_OG} Let $\zeta\in\mathbb{D}$. 
Let $f\in\mathcal{U}(\mathbb{D},\zeta)$. Let $K\subset\mathbb{C}\setminus\mathbb{D}$ be a compact set with connected complement and 
let $h\in A(K)$. Then there exist two sequences of positive integers $(p_m),$ $(q_m)$ such that 
\begin{enumerate}
\item the Taylor series of $f$ at $\zeta$ has Ostrowski-gaps $(p_m,q_m)$,
\item and $\displaystyle\sup_{z\in K}\vert S_{p_m}(f,\zeta)(z)-h(z)\vert\rightarrow 0$, as $m\rightarrow +\infty$.
\end{enumerate} 
\end{lemma}

Combined Lemma \ref{lemme_OG} with the definition of universal Taylor series we immediately deduce the following useful lemma. 

\begin{lemma}\label{lemme_OG2} Let $\zeta\in\mathbb{D}$. 
Let $f$ be in $H(\mathbb{D})$ and suppose that the Taylor series of $f$ at $0$ has Ostrowski-gaps $(p_m,q_m)$. 
Then for every sequence $(r_m)$ with $p_m<r_m\leq q_m$ the difference between partial sums $S_{r_m}(f,\zeta)(z)-S_{p_m}(f,\zeta)(z)$ converges 
uniformly to zero (as $m\rightarrow +\infty$) on compact sets of $\mathbb{C}.$
\end{lemma}

In connection with the Ostrowski-gaps, we also have the following result (\cite[Theorem 1]{Luh1986} or \cite[Lemma 9.2]{MNadv}). 
The published proof uses the Hadamard three-circle theorem. Here we give an elementary proof of the result.    

\begin{lemma}\label{lemme_npf} Let $f$ be in $H(\mathbb{D})$ and $\zeta_0,\zeta\in\mathbb{D}$. 
Suppose that the Taylor series of $f$ at $\zeta_0$ has Ostrowski-gaps $(p_m,q_m)$. 
Then the difference $S_{p_k}(f,\zeta)(z)-S_{p_k}(f,\zeta_0)(z)$ converges to zero 
(as $k\rightarrow +\infty$) uniformly on compact sets of $\mathbb{D}\times\mathbb{C}$ ($\zeta\in \mathbb{D}$, $z\in\mathbb{C}$). 
\end{lemma}

\begin{proof} Without loss of generality, we can suppose that $\zeta_0=0$ (and $\zeta\ne\zeta_0$!). 
Let $K\subset\mathbb{C}$ and $L\subset\mathbb{D}$ be fixed compact sets. 
We define 
$$r=\sup_{\zeta\in L}\vert \zeta\vert\hbox{ and }M>\max(2,\sup\{\vert z-\zeta\vert :z\in K, \zeta\in L\}).$$ 
Let us choose $0<\varepsilon <1$ such that 
$$r(1+\varepsilon)<1\hbox{ and }2\varepsilon M<1.$$ 
We write, for all $\vert z\vert<1$, $f(z)=\displaystyle\sum_{k=0}^{+\infty}a_k z^k$. Thus, for all $j\geq 0,$ we have 
$f^{(j)}(\zeta)= j! \displaystyle\sum_{k=j}^{+\infty}a_k {k\choose j} \zeta^{k-j}.$ 
Using the equality $z^k=\displaystyle\sum_{l=0}^k a_k{k\choose l}(z-\zeta)^l\zeta^{k-l}$, we get, for all $\zeta\in L,$ 
\begin{equation}\label{LN_equ1}\begin{array}{rcl}\displaystyle S_{p_m}(f,\zeta)(z)-S_{p_m}(f)(z)&=&\displaystyle 
\sum_{j=0}^{p_m}\left(\sum_{k=j}^{+\infty}a_k{k\choose j}\zeta^{k-j}\right)(z-\zeta)^j-
\sum_{j=0}^{p_m}\left(\sum_{k=j}^{p_m}a_k{k\choose j}\zeta^{k-j}\right)(z-\zeta)^j\\
&=&\displaystyle\sum_{j=0}^{p_m}\left(\sum_{k=1+p_m}^{q_m}a_k{k\choose j}\zeta^{k-j}\right)(z-\zeta)^j\\&&\ \displaystyle+
\sum_{j=0}^{p_m}\left(\sum_{k=1+q_m}^{+\infty}a_k{k\choose j}\zeta^{k-j}\right)(z-\zeta)^j
\\&:=&
A_{1,m}(\zeta,z)+A_{2,m}(\zeta,z)\end{array}
\end{equation} 
Now we are going to estimate the series $A_{1,m}(\zeta,z)$and $A_{2,m}(\zeta,z)$. 
On one hand, by the triangle inequality and the inequality ${k\choose j}\leq 2^k$, we get
$$\sup_{\zeta\in L}\sup_{z\in K }\vert A_{1,m}(\zeta,z)\vert\leq
\sum_{j=0}^{p_m}\left(\sum_{k=1+p_m}^{q_m}\vert a_k\vert 2^kr^{k-j}\right) M^j.$$
Using the Ostrowski-gaps, one can find $m_1$ such that, for all $m\geq m_1$ and $1+p_m\leq k\leq q_m$,  
$\vert a_k\vert <\varepsilon^k.$ We deduce, using the property $M>2r$,
\begin{equation}\label{LN_equ2}
\sup_{\zeta\in L}\sup_{z\in K }\vert A_{1,m}(\zeta,z)\vert\leq
\sum_{j=0}^{p_m}\left(\sum_{k=1+p_m}^{q_m}(2\varepsilon r)^k\right) \left(\frac{M}{r}\right)^j\leq 
\frac{1}{1-2\varepsilon r}(2\varepsilon M)^{1+p_m}.
\end{equation}
On the other hand, we need the following classical inequality
\begin{equation}\label{LN_nouv_equ2}{k\choose j}\leq \frac{k^j}{j!}\leq \left(\frac{ek}{j}\right)^j
\end{equation} 
and the fact that, for $t< k$, the function $t\mapsto \left(\frac{ek}{t}\right)^t$ is increasing. 
It follows
\begin{equation}\label{LN_nouv_equ3}\sup_{\zeta\in L}\sup_{z\in K }\vert A_{2,m}(\zeta,z)\vert\leq
\sum_{j=0}^{p_m}\left(\sum_{k=1+q_m}^{+\infty}\vert a_k\vert e^{p_m}\left(\frac{k}{p_m}\right)^{p_m}r^{k-j}\right)M^j.
\end{equation}
Since $p_m/q_m\rightarrow 0,$ as $m$ tends to infinity, we have, for all $k\geq 1+q_m$ and $0\leq j\leq p_m$,
$$\begin{array}{rcl}\displaystyle 
e^{\frac{p_m}{k-j}}e^{\frac{p_m}{k-j}\log (\frac{k}{p_m})}&\leq &\displaystyle 
e^{\frac{p_m}{1+q_m-p_m}}e^{\frac{p_m}{k(1-j/k)}\log(\frac{k}{p_m})}\\&\leq&\displaystyle 
e^{\frac{p_m}{1+q_m-p_m}}e^{-\frac{p_m}{k(1-\frac{p_m}{1+q_m})}\log(\frac{p_m}{k})}\rightarrow 1\hbox{ as }m\rightarrow +\infty.
\end{array}$$
Therefore, since $\limsup \vert a_k\vert^{1/k}\leq 1$ and $p_m/q_m\rightarrow 0,$ as $m$ tends to infinity, there exists a positive integer 
$m_2$ such that for all $m\geq m_2$, $k\geq 1+q_m$ and $0\leq j\leq p_m$, 
$$\left(\frac{ke}{p_m}\right)^{p_m/(k-j)}\vert a_k\vert^{1/(k-j)}r\leq (1+\varepsilon)r.$$
Since the choice of $\varepsilon$ ensures $(1+\varepsilon)r<1$ and $M>(1+\varepsilon)r$, we deduce that for all $m\geq m_2$
\begin{equation}\label{LN_equ3}\begin{array}{rcl}\displaystyle
\sup_{\zeta\in L}\sup_{z\in K }\vert A_{2,m}(\zeta,z)\vert&\leq&\displaystyle 
\sum_{j=0}^{p_m}\left(\sum_{k=1+q_m}^{+\infty}(1+\varepsilon)^{k-j}r^{k-j}\right)M^j\\&
\leq&\displaystyle \frac{\left((1+\varepsilon)r\right)^{1+q_m}}{1-(1+\varepsilon)r}\sum_{j=0}^{p_m}\left(\frac{M}{(1+\varepsilon)r}\right)^{j}\\
&\leq&\displaystyle \frac{1+p_m}{1-(1+\varepsilon)r}\left(\frac{M}{(1+\varepsilon)r}\right)^{p_m}\left((1+\varepsilon)r\right)^{1+q_m}.\end{array}
\end{equation}
Finally taking into account $p_m/q_m\rightarrow 0$ again and combining (\ref{LN_equ1}) with (\ref{LN_equ2}) and (\ref{LN_equ3}), we derive 
$$\sup_{\zeta\in L}\sup_{z\in K }\left\vert S_{p_m}(f,\zeta)(z)-S_{p_m}(f)(z)\right\vert\rightarrow 0,\hbox{ as }m\rightarrow +\infty.$$
\end{proof}

Notice that the statements of \cite[Theorem 1]{Luh1986} or \cite[Lemma 9.2]{MNadv} are given in the more general case where 
the open unit disc $\mathbb{D}$ is replaced by a simply connected domain $\Omega$ with $\Omega\subset\mathbb{C}$. Obviously the same proof 
does the job with easy modifications.\\ 

On the other hand, we will need a specific version of Bernstein-Walsh theorem. It is a polynomial approximation theorem 
which allows in some sense to control both the degree and the valuation of the polynomials. This elegant statement was given in 
\cite{CT}. For given sequence $(x_n)$, $(y_n)$ of positive real numbers, the notation $y_n=O(x_n)$ means that the sequence 
$(y_n/x_n)$ is bounded.
   
\begin{theorem}\label{thmCT}\cite[Theorem 2.1]{CT} Let $(\sigma_n),$ $(\tau_n)$ be strictly increasing sequences 
of positive integers, let $K\subset\mathbb{C}\setminus\mathbb{D}$ be a compact set with connected complement 
and let $r\in (0,1)$. If $1\leq \tau_n/\sigma_n\rightarrow +\infty$ as $n\rightarrow +\infty$ and if $U$ 
is open in $\mathbb{C}$ with $K\subset U$, then there is $\theta\in (0,1)$ so that for every $h\in H(U)$ there 
exists a sequence of polynomials $(P_n)$ of the form\quad 
$$P_n(z)=\displaystyle\sum_{k=\sigma_n}^{\tau_n}c_{n,k}z^k$$ 
\begin{flalign*}&\hbox{with}\quad \quad \quad\quad\quad\quad
\sup_{z\in K}\vert h(z)-P_n(z)\vert =O(\theta^{\tau_n})\quad\hbox{ and }\quad 
\sup_{\vert z\vert\leq r}\vert P_n(z)\vert=O(\theta^{\tau_n}).&
\end{flalign*}
\end{theorem}

\subsection{Proof of Theorem \ref{main_thm}} 
In order to simplify the notations, we write the proof for the class $\mathcal{U}(\mathbb{D},0)$. The proof works along the same lines in the case 
of the class $\mathcal{U}(\mathbb{D},\zeta),$ $\zeta\in\mathbb{D}$. 

\begin{theorem}\label{main_thm2} Let $\mu=(\mu_n)$ be a strictly increasing sequence of positive integers. The following assertions 
are equivalent:
\begin{enumerate}[(i)] 
\item\label{1ass} $\mathcal{U}(\mathbb{D},0)=\mathcal{U}^{(\mu)}(\mathbb{D},0)$
\item\label{2ass} $\displaystyle\limsup_{n\rightarrow +\infty}\left(\frac{\mu_{n+1}}{\mu_n}\right)<+\infty.$
\end{enumerate}
\end{theorem}

\begin{proof}  $(\ref{2ass})\Rightarrow(\ref{1ass})$: assume that $\displaystyle\limsup_{n\rightarrow +\infty}\left(\frac{\mu_{n+1}}{\mu_n}\right)<+\infty.$ 
Since the inclusion $\mathcal{U}^{(\mu)}(\mathbb{D},0)\subset \mathcal{U}(\mathbb{D},0)$ is obvious, it suffices to prove 
$\mathcal{U}(\mathbb{D},0)\subset \mathcal{U}^{(\mu)}(\mathbb{D},0)$. 
Let also $f$ be in $\mathcal{U}(\mathbb{D},0)$. Thus according to Lemma \ref{lemme_OG} 
for all compact subset $K\subset\mathbb{C}\setminus\mathbb{D}$ with connected complement 
and for all $h\in A(K),$ there exists two sequences of positive integers $(p_m),$ $(q_m)$ such that 
\begin{enumerate}
\item the Taylor series of $f$ at $0$ has Ostrowski-gaps $(p_m,q_m)$,
\item and $\displaystyle\sup_{z\in K}\vert S_{p_m}(f)(z)-h(z)\vert\rightarrow 0$, as $m\rightarrow +\infty$.
\end{enumerate}
Since we have both 
$$\displaystyle\limsup_{n\rightarrow +\infty}\left(\frac{\mu_{n+1}}{\mu_n}\right)<+\infty\quad\hbox{ and }\quad 
\displaystyle\frac{q_m}{p_m}\rightarrow +\infty,\hbox{ as }m\rightarrow +\infty,$$
there exists $N\in\mathbb{N},$ such that for all $m\geq N$ one can find $\mu_{j_m}$ with 
$p_m\leq \mu_{j_m}<q_m$. Hence we apply Lemma \ref{lemme_OG2} to obtain  
$$\sup_{z\in K}\vert S_{p_m}(f)(z)-S_{\mu_{j_m}}(f)(z)\vert\rightarrow 0,\hbox{ as }m\rightarrow +\infty.$$
From the triangle inequality we get
$$\sup_{z\in K}\vert S_{\mu_{j_m}}(f)(z)-h(z)\vert\rightarrow 0,\hbox{ as }m\rightarrow +\infty.$$
This implies $f\in\mathcal{U}^{(\mu)}(\mathbb{D},0)$. \\

$(\ref{1ass})\Rightarrow(\ref{2ass})$: to do this, we assume that 
$\displaystyle\limsup_{n\rightarrow +\infty}\left(\frac{\mu_{n+1}}{\mu_n}\right)=+\infty$ 
and it suffices to exhibit an universal series $f\in \mathcal{U}(\mathbb{D},0)$ such that $f\notin \mathcal{U}^{(\mu)}(\mathbb{D},0)$. 
By hypothesis there exists an increasing subsequence of positive integers $(n_j)$ 
such that $\displaystyle\frac{\mu_{n_j+1}}{\mu_{n_j}}\rightarrow +\infty$ as $j\rightarrow +\infty$. We set, for all $j\geq 1$,  
\begin{equation}\label{equa_suites_def}u_j= \mu_{n_j}+1,\quad w_j= \mu_{n_j+1}\quad\hbox{ and }v_j=\lfloor \sqrt{u_j w_j}\rfloor.
\end{equation}
Clearly there exists $N_0\in\mathbb{N}$ such that, for all 
$j\geq N_0$, $u_j<v_j<w_j$ and 
$$\frac{v_j}{u_j}=\frac{\lfloor \sqrt{u_j w_j}\rfloor}{u_j}\leq
\sqrt{\frac{\mu_{n_j+1}}{\mu_{n_j}}}\rightarrow +\infty\hbox{ as }
j\rightarrow +\infty$$
and
$$\frac{w_j}{1+v_j}=\frac{w_j}{1+\lfloor \sqrt{u_j w_j}\rfloor}\leq
\frac{\mu_{n_j+1}}{\sqrt{(1+\mu_{n_j})\mu_{n_j+1}}}\rightarrow +\infty\hbox{ as }
j\rightarrow +\infty.$$
Let $(f_q)$ be an enumeration of all the polynomials with coefficients in $\mathbb{Q}+i\mathbb{Q}$. Let 
$(K_m)$ be a sequence of compact sets with connected complement and $K_m\cap\mathbb{D}=\emptyset$ for every 
$m\in\mathbb{N}$ such that for every compact set $K\subset\mathbb{C}\setminus\mathbb{D}$ with connected complement 
there exists $n\in\mathbb{N}$ such that $K\subset K_n$ (see \cite[Lemma 2.1]{NestorF}). 
We consider an enumeration $(K_{m_s},f_{q_s})$, $s=1,2,\dots$, of all couples $(K_m,f_q),$ $m,q=1,2,\dots$. Let also $(r_l)$ be an increasing sequence 
of real numbers with $0<r_l<1$ and $r_l\rightarrow 1,$ as $l\rightarrow +\infty$. We fix $z_0\in\mathbb{C}\setminus\mathbb{D}$. For all 
$l\geq 1$, we set $\tilde{K}_l=K_l\cup\{z_0\}$ and note that $\tilde{K}_l$ is a compact set with 
connected complement and $\tilde{K}_l\subset\mathbb{C}\setminus\mathbb{D}$. First we deal with $\tilde{K}_{m_1},$ $f_{q_1}$ and $r_1$. 
By applying Theorem \ref{thmCT}, we find $j_1\in\mathbb{N}$ and polynomials 
$$P_1(z)=\displaystyle\sum_{k=u_{j_1}}^{v_{j_1}}c_{j_1,k}z^k,\quad \tilde{P}_1(z)=\displaystyle\sum_{k=v_{j_1}+1}^{w_{j_1}}c_{j_1,k}z^k$$ 
such that 
$$\sup_{z\in \tilde{K}_{m_1}}\vert P_1(z)-f_{q_1}(z)\vert <\frac{1}{2^2},\quad \sup_{\vert z\vert\leq r_1}\vert P_1(z)\vert\leq \frac{1}{2^2}$$
and
$$\sup_{z\in \tilde{K}_{m_1}}\vert \tilde{P}_1(z)+f_{q_1}(z)\vert< \frac{1}{2^2}, \quad \sup_{\vert z\vert\leq r_1}\vert \tilde{P}_1(z)\vert\leq \frac{1}{2^2}.$$
Observe that we have, by the triangle inequality,
$$\begin{array}{rcl}\displaystyle\sup_{z\in \tilde{K}_{m_1}}\vert {P}_1(z)+ \tilde{P}_1(z)\vert&\leq&
\displaystyle \sup_{z\in \tilde{K}_{m_1}}\vert {P}_1(z)-f_{q_1}(z)\vert + 
\sup_{z\in \tilde{K}_{m_1}}\vert \tilde{P}_1(z)+f_{q_1}(z)\vert\\& <&\displaystyle\frac{2}{2^2}.\end{array}$$
Further we find $j_2\in\mathbb{N}$ with $j_2>j_1$ and polynomials 
$$P_2(z)=\displaystyle\sum_{k=u_{j_2}}^{v_{j_2}}c_{j_2,k}z^k,\quad  \tilde{P}_2(z)=\displaystyle\sum_{k=v_{j_2}+1}^{w_{j_2}}c_{j_2,k}z^k$$ 
such that 
$$\sup_{z\in \tilde{K}_{m_2}}\vert P_1(z)+\tilde{P}_1(z)+P_2(z)-f_{q_2}(z)\vert <\frac{1}{3^2},\quad \sup_{\vert z\vert\leq r_2}\vert P_2(z)\vert\leq \frac{1}{3^2}$$
and
$$\sup_{z\in \tilde{K}_{m_2}}\vert \tilde{P}_2(z)+f_{q_2}(z)\vert< \frac{1}{2^2}, \quad \sup_{\vert z\vert\leq r_2}\vert \tilde{P}_1(z)\vert\leq \frac{1}{3^2}.$$
Observe that we have, by the triangle inequality,
$$\begin{array}{rcl}\displaystyle\sup_{z\in \tilde{K}_{m_2}}\left\vert \sum_{i=1}^2\left({P}_i(z)+ \tilde{P}_i(z)\right)\right\vert&\leq &\displaystyle
\sup_{z\in \tilde{K}_{m_2}}\vert P_1(z)+ \tilde{P}_1(z)+{P}_2(z)-f_{q_2}(z)\vert + 
\sup_{z\in \tilde{K}_{m_2}}\vert \tilde{P}_2(z)+f_{q_2}(z)\vert \\&<&\displaystyle\frac{2}{3^2}.\end{array}$$
We argue by induction. Suppose that for a natural number $s\geq 2$ we have already defined integers 
$$j_1<j_2<\dots<j_s$$
and polynomials $P_i,\tilde{P}_i$, $i=1,\dots,s$ such that 
\begin{equation}\label{equaf1}
P_i(z)=\sum_{k=u_{j_i}}^{v_{j_i}}c_{j_i,k}z^k,\quad \tilde{P}_i(z)=\sum_{k=v_{j_i}+1}^{w_{j_i}}c_{j_i,k}z^k
\end{equation}
with
\begin{equation}\label{equaf2}\sup_{z\in \tilde{K}_{m_s}}\left\vert \sum_{i=1}^{s-1}P_i(z) +P_{s}(z)-f_{q_s}(z)\right\vert <\frac{1}{(s+1)^2},\quad 
\sup_{\vert z\vert\leq r_s}\vert P_s(z)\vert\leq \frac{1}{(s+1)^2}
\end{equation}
\begin{equation}\label{equaf3}\sup_{z\in \tilde{K}_{m_s}}\vert \tilde{P}_s(z)+f_{q_s}(z)\vert< \frac{1}{(s+1)^2}, \quad \sup_{\vert z\vert\leq r_s}\vert \tilde{P}_s(z)\vert\leq \frac{1}{(s+1)^2}.
\end{equation}
and
\begin{equation}\label{equaf4}\displaystyle\sup_{z\in \tilde{K}_{m_s}}\left\vert \sum_{i=1}^s({P}_i(z)+ \tilde{P}_i(z))\right\vert<\frac{2}{(s+1)^2}.
\end{equation}
Using Theorem \ref{thmCT}, we find $j_{s+1}\in\mathbb{N}$ with $j_{s+1}>j_s$ and polynomials 
$$P_{s+1}(z)=\sum_{k=u_{j_{s+1}}}^{v_{j_{s+1}}}c_{j_{s+1},k}z^k,\quad \tilde{P}_{s+1}(z)=\sum_{k=v_{j_{s+1}}+1}^{w_{j_{s+1}}}c_{j_{s+1},k}z^k$$
such that 
$$\sup_{z\in \tilde{K}_{m_{s+1}}}\left\vert \sum_{i=1}^{s}P_i(z) +P_{s+1}(z)-f_{q_{s+1}}(z)\right\vert <\frac{1}{(s+2)^2},\quad 
\sup_{\vert z\vert\leq r_{s+1}}\vert P_{s+1}(z)\vert\leq \frac{1}{(s+2)^2}$$
and
$$\sup_{z\in \tilde{K}_{m_{s+1}}}\vert \tilde{P}_{s+1}(z)+f_{q_{s+1}}(z)\vert< \frac{1}{(s+2)^2}, \quad 
\sup_{\vert z\vert\leq r_{s+1}}\vert \tilde{P}_{s+1}(z)\vert\leq \frac{1}{(s+2)^2}.$$
Thus from the triangle inequality we get 
$$\begin{array}{rcl}\displaystyle\sup_{z\in \tilde{K}_{m_{s+1}}}\left\vert \sum_{i=1}^{s+1}({P}_i(z)+ \tilde{P}_i(z))\right\vert&\leq &\displaystyle
\sup_{z\in \tilde{K}_{m_{s+1}}}\left\vert \sum_{i=1}^{s}\left({P}_i(z)+\tilde{P}_i(z)\right)+{P}_{s+1}(z)-f_{q_{s+1}}(z)\right\vert\\&&\ 
\displaystyle + 
\sup_{z\in \tilde{K}_{m_{s+1}}}\vert \tilde{P}_{s+1}(z)+f_{q_{s+1}}(z)\vert \\&<&\displaystyle\frac{2}{(s+1)^2}.\end{array}$$
The induction is valid. Finally we set 
$$f(z)=\sum_{i\geq 1}\left(P_{i}(z)+\tilde{P}_i(z)\right).$$
Thanks to the second inequalities of $(\ref{equaf2})$ and $(\ref{equaf3})$ this series converges on all compact subsets of $\mathbb{D}$. So $f\in H(\mathbb{D})$. 
The first inequality of $(\ref{equaf2})$ ensures that $f\in\mathcal{U}(\mathbb{D},0)$. Indeed, let 
$K\subset\mathbb{C}\setminus\mathbb{D}$ be a compact set with connected 
complement and $h\in A(K)$. Set $\varepsilon>0$ and $s_0\in\mathbb{N}$ with $1/(s_0+1)^2<\varepsilon/2$. By hypothesis, one can find a positive integer $p>s_0$ 
such that $K\subset K_{m_p}$ and $\sup_{K}\vert h-f_{q_p}\vert<\varepsilon/2$. Thus from $(\ref{equaf2})$ we get 
$$\sup_{z\in K}\vert S_{v_{j_p}}(f)(z)-h(z)\vert\leq \sup_{z\in K_{m_p}}\vert S_{v_{j_p}}(f)(z)-f_{q_p}(z)\vert + \sup_{z\in K_{m_p}}\vert h(z)-f_{q_p}(z)\vert <
\frac{1}{(p+1)^2}+\frac{\varepsilon}{2}<\varepsilon.$$
Moreover observe that the property (\ref{equa_suites_def}) 
implies that
$$\{\mu_n;n\geq n_{j_1}+1\}\cap \left(\cup_{s\geq 1}\{u_{j_s},u_{j_s}+1,\dots,w_{j_s}-1\}\right)=\emptyset.$$
Therefore the equation (\ref{equaf4}) guarantees that, for all $n\in\mathbb{N},$
$$\left\vert S_{\mu_n}(f)(z_0)\right\vert\leq 2\sum_{s\geq 1}\frac{1}{(s+1)^2}=\frac{\pi^2}{3}-2.$$
From this last inequality we easily deduce that $f\notin\mathcal{U}^{(\mu)}(\mathbb{D},0)$. 
\end{proof}

\subsection{Universal Taylor series and center independence} Let us recall that the classes $\mathcal{U}(\mathbb{D},0)$ and $\mathcal{U}(\mathbb{D},\zeta)$ coincide 
for all $\zeta\in\mathbb{D}$ \cite{GLM,MNadv}. The proof is based on Lemma \ref{lemme_OG}, Lemma \ref{lemme_OG2} and Lemma \ref{lemme_npf} 
which asserts that if a Taylor series $f$ has Ostrowski-gaps $(p_m,q_m)$ then the difference 
$S_{p_m}(f,\zeta)(z)-S_{p_m}(f,0)(z)$ converges to zero (as $m\rightarrow +\infty$) uniformly on compact sets of 
$\mathbb{D}\times \mathbb{C}$ ($\zeta\in\mathbb{D}$, $z\in\mathbb{C}$). 
But if you can choose $q_m\in\mu$, there is no evidence that you can choose $p_m\in\mu$. 
Thus it is not clear that we have $\mathcal{U}^{(\mu)}(\mathbb{D},0)=\mathcal{U}^{(\mu)}(\mathbb{D},\zeta)$. 
Nevertheless Theorem \ref{main_thm} immediately leads to the 
following corollary.

\begin{corollary}\label{main_thm3} 
Let $\zeta\in\mathbb{D}$. Let $\mu=(\mu_n)$ be a strictly increasing sequence of positive integers with 
$\displaystyle\limsup_{n\rightarrow +\infty}\left(\frac{\mu_{n+1}}{\mu_n}\right)<+\infty.$ Then we have 
$\mathcal{U}^{(\mu)}(\mathbb{D},\zeta)=\mathcal{U}^{(\mu)}(\mathbb{D},0)$.
\end{corollary}

\begin{proof} We know that, for all $\zeta\in\mathbb{D}$, $\mathcal{U}(\mathbb{D},\zeta)=\mathcal{U}(\mathbb{D},0)$ \cite{GLM, MNadv}. 
Assume that $\displaystyle\limsup_{n\rightarrow +\infty}\left(\frac{\mu_{n+1}}{\mu_n}\right)<+\infty.$ Hence Theorem 
\ref{main_thm} ensures that $\mathcal{U}^{(\mu)}(\mathbb{D},\zeta)=\mathcal{U}(\mathbb{D},\zeta)$ and $\mathcal{U}^{(\mu)}(\mathbb{D},0)=\mathcal{U}(\mathbb{D},0)$. 
We get $\mathcal{U}^{(\mu)}(\mathbb{D},\zeta)=\mathcal{U}^{(\mu)}(\mathbb{D},0)$.
\end{proof}

Corollary \ref{main_thm3} covers all the known examples of sequences $\mu$ such that $\mathcal{U}^{(\mu)}(\mathbb{D},\zeta)=\mathcal{U}^{(\mu)}(\mathbb{D},0)$ 
\cite{Vlachou1, Vlachou2}. Moreover Corollary \ref{main_thm3} is optimal in the following sense. 

\begin{theorem}\label{rec_thm} 
Let $\zeta\in\mathbb{D}$, $\zeta\ne 0$. Let $\mu=(\mu_n)$ be a strictly increasing sequence of positive integers with 
$\displaystyle\limsup_{n\rightarrow +\infty}\left(\frac{\mu_{n+1}}{\mu_n}\right)=+\infty.$ Then we have 
$\mathcal{U}^{(\mu)}(\mathbb{D},\zeta)\ne\mathcal{U}^{(\mu)}(\mathbb{D},0)$.
\end{theorem}

\begin{proof} Let $\zeta\in\mathbb{D}$, $\zeta\ne 0$. 
We are going to build an universal series $f\in \mathcal{U}^{(\mu)}(\mathbb{D},0)$ such that 
$f\notin \mathcal{U}^{(\mu)}(\mathbb{D},\zeta)$. Let us consider 
$z_0$ with $\vert z_0\vert=1$ such that $\vert z_0-\zeta\vert=d(\zeta,\partial \mathbb{D}),$ the distance between $\zeta$ and the unit circle. By hypothesis there exists an increasing sequence of integers $(n_j)$ such that 
$$\frac{\mu_{n_j+1}}{\mu_{n_j}}\rightarrow +\infty,\hbox{ as }j\rightarrow +\infty.$$ 
We set, for all $j\geq 1$,  
\begin{equation}\label{20_equa_suites_def}u_j= \mu_{n_j}+1,\quad w_j= \mu_{n_j+1}\quad\hbox{ and }
v_j=\lfloor \sqrt{u_j w_j}\rfloor.\end{equation}
As in the proof of Theorem \ref{main_thm2}, we have for all $j$ large enough $u_j<v_j<w_j$ and 
\begin{equation}\label{21_equa_suites_def}\frac{v_j}{u_j}\rightarrow +\infty\hbox{ and } \frac{w_j}{1+v_j}\rightarrow 
+\infty\hbox{ as }
j\rightarrow +\infty.\end{equation}
Let us also consider an enumeration $(K_{m_s},f_{q_s})$, $s=1,2,\dots$, of all couples 
$(K_m,f_q),$ $m,q=1,2,\dots$, where $(f_q)$ is an enumeration of all the polynomials with coefficients 
in $\mathbb{Q}+i\mathbb{Q}$ and
$(K_m)$ is a sequence of compact sets with connected complement and $K_m\cap\mathbb{D}=\emptyset$ 
such that for every compact set $K\subset\mathbb{C}\setminus\mathbb{D}$ with connected complement 
there exists $n\in\mathbb{N}$ such that $K\subset K_n$. Let also $(r_l)$ be an increasing sequence 
of real numbers with $0<r_l<1$ and $r_l\rightarrow 1,$ as $l\rightarrow +\infty$. 
For all $l\geq 1$, we set $\tilde{K}_l=K_l\cup\{z_0\}$. Observe that $\tilde{K}_l$ is a compact set with 
connected complement and $\tilde{K}_l\subset\mathbb{C}\setminus\mathbb{D}$. First we deal with 
$\tilde{K}_{m_1},$ $f_{q_1}$ and $r_1$. By applying Theorem \ref{thmCT}, we find $j_1\in\mathbb{N}$ and polynomial 
$$P_1(z)=\displaystyle\sum_{k=1+v_{j_1}}^{w_{j_1}}a_{k}z^k$$ such that, for all $j\geq j_1,$
\begin{equation}\label{cons2_equ1}
2\vert z_0^j-(z_0-\zeta)^{j}\vert =2 (1-(1-\vert \zeta\vert)^j)>1
\end{equation}
and 
$$\vert f_{q_1}(z_0)\vert\leq w_{j_1},$$
$$\sup_{\tilde{K}_{m_1}}\vert  P_1(z)-f_{q_1}(z)\vert<\frac{1}{2}\hbox{ and }
\sup_{\vert z\vert\leq r_1}\vert P_1(z)\vert<\frac{1}{2}.$$
By the triangle inequality we have  
\begin{equation}\label{cons2_equ12}\vert P_1(z_0)\vert\leq \vert P_1(z_0)-f_{q_1}(z_0)\vert+\vert f_{q_1}(z_0)\vert\leq 
\frac{1}{2}+w_{j_1}\leq 1+w_{j_1}.
\end{equation}
We define 
$$a_{1+w_{j_1}}=-\frac{P_1(z_0)}{z_0^{1+w_{j_1}}-(z_0-\zeta)^{1+w_{j_1}}}.$$
Thanks to (\ref{cons2_equ1}) and (\ref{cons2_equ12}), we have 
$$\vert a_{1+w_{j_1}}\vert \leq 2 (1+w_{j_1}).$$
Then we set 
$$R_1(z)= a_{1+w_{j_1}}z^{1+w_{j_1}}.$$
By induction we construct an increasing sequence of integers $(j_l)$, two sequences of polynomials 
$(P_l)$ and $(R_l)$ such that, for all $l\geq 2$,
\begin{equation}\label{cons2_Luh}\frac{\mu_{n_{j_l}}}{1+\mu_{n_{j_{l-1}}}}\rightarrow +\infty,\hbox{ as }l\rightarrow +\infty
\end{equation}
$$\vert f_{q_l}(z_0)\vert\leq w_{j_l},$$
\begin{equation}\label{cons2_equ20}P_l(z)=\displaystyle\sum_{k=1+v_{j_l}}^{w_{j_l}}a_{k}z^k,\quad 
R_l(z)=a_{1+w_{j_l}}z^{1+w_{j_l}}
\end{equation}
with
\begin{equation}\label{cons2_equ21}\sup_{\tilde{K}_{m_l}}\left\vert  P_l(z)+\sum_{i=1}^{l-1}(P_i(z)+R_i(z))-f_{q_l}(z)\right\vert<\frac{1}{2^l}\hbox{ and }
\sup_{\vert z\vert\leq r_l}\vert P_l(z)\vert<\frac{1}{2^l},
\end{equation}
and
\begin{equation}\label{cons2_equ22}a_{1+w_{j_l}}=-\frac{P_l(z_0)+\displaystyle\sum_{i=1}^{l-1}(P_i(z_0)+R_i(z_0))}
{z_0^{1+w_{j_l}}-(z_0-\zeta)^{1+w_{j_l}}}.
\end{equation}
By the triangle inequality we have  
$$\vert P_l(z_0)\vert\leq 
\frac{1}{2^l}+w_{j_l}+\leq 1+w_{j_l}$$
and, taking into account (\ref{cons2_equ1}), we deduce
\begin{equation}\label{cons2_equ23}
\vert a_{1+w_{j_l}}\vert \leq 2 (1+w_{j_l}).
\end{equation}
We also define, for all $l\geq 2$, 
\begin{equation}\label{cons2_equ24_bis}
a_k=0\hbox{ for }k=2+w_{j_{l-1}},\dots,v_{j_l}.
\end{equation}
By construction we get, for all $l\geq 1$, using (\ref{cons2_equ22}), 
\begin{equation}\label{cons2_equ24}
\begin{array}{rcl}\displaystyle
\sum_{j=0}^{w_{j_l}}\left(\sum_{k=j}^{1+w_{j_l}}a_k{ k\choose j}\zeta^{k-j}\right) (z_0-\zeta)^j&=&\displaystyle
\sum_{j=0}^{w_{j_l}}\left(\sum_{k=j}^{w_{j_l}}a_k{ k\choose j}\zeta^{k-j}\right) (z_0-\zeta)^j \\&&\displaystyle
+\ 
a_{1+w_{j_l}}\sum_{j=0}^{w_{j_l}}{1+w_{j_l}\choose j}\zeta^{1+w_{j_l}-j} (z_0-\zeta)^j\\&
=&\displaystyle\sum_{i=1}^{l-1} (P_i(z_0)+R_i(z_0)) +P_l(z_0)\\&& +\ a_{1+w_{j_l}} 
\left(z_0^{1+w_{j_l}}-(z_0-\zeta)^{1+w_{j_l}}\right)\\&=&0.\end{array}
\end{equation}
Now we consider 
$$f(z)=\sum_{i=1}^{+\infty}(P_i(z)+R_i(z)):=\sum_{k\geq 0} a_k z^k.$$ 
Thanks to the estimates (\ref{cons2_equ21}) and (\ref{cons2_equ23}) the series $f$ belongs to $H(\mathbb{D})$. 
The equations (\ref{cons2_equ20}) and (\ref{cons2_equ21}) ensure that, for all $l\geq 1$,  
$$\sup_{z\in K_{m_l}}\vert S_{w_{j_l}}(f)(z)-f_{q_l}(z)\vert < \frac{1}{2^l}.$$
Since, for all $l\geq 1$, $w_{j_l}=\mu_{1+n_{j_l}}$, we deduce that 
$f$ belongs to $\mathcal{U}^{(\mu)}(\mathbb{D},0)$. To finish the 
proof, we are going to prove that $f\notin \mathcal{U}^{(\mu)}(\mathbb{D},\zeta)$. Notice that we have 
$$S_{m}(f,\zeta)(z_0)= \sum_{j=0}^{m}\left(\sum_{k=j}^{+\infty}a_k{ k\choose j}\zeta^{k-j}\right) (z_0-\zeta)^j.$$
First we deal with the case $m=w_{j_l}$. We have, thanks to the properties 
(\ref{cons2_equ24_bis}) and (\ref{cons2_equ24}),
\begin{equation}\label{cons2_equ30}S_{w_{j_l}}(f,\zeta)(z_0)=
\sum_{j=0}^{w_{j_l}}\left(\sum_{k=1+v_{j_{l+1}}}^{+\infty}a_k{ k\choose j}\zeta^{k-j}\right) (z_0-\zeta)^j.
\end{equation}
Using the equation (\ref{cons2_Luh}), we argue as in the estimates (\ref{LN_nouv_equ2}), (\ref{LN_nouv_equ3}) and 
(\ref{LN_equ3}) of the proof of Lemma \ref{lemme_npf} to obtain 
\begin{equation}\label{cons2_nonuniv}
S_{w_{j_l}}(f,\zeta)(z_0)\rightarrow 0,\hbox{ as }l\rightarrow +\infty.
\end{equation}

Now, for all 
$l\geq 1$  and  for all positive integer $m\in [1+w_{j_l},u_{j_{l+1}}]$, we can write, thanks to (\ref{cons2_equ24_bis}) 
and (\ref{cons2_equ24}) again,
$$\begin{array}{rcl}S_{m}(f,\zeta)(z_0)&= &\displaystyle
\sum_{j=0}^{w_{j_l}}\left(\sum_{k=1+v_{j_{l+1}}}^{+\infty}a_k{ k\choose j}\zeta^{k-j}\right) (z_0-\zeta)^j\\&&
+\ \displaystyle\sum_{j=1+w_{j_l}}^m\left(\sum_{k=j}^{+\infty}a_k{ k\choose j}\zeta^{k-j}\right) (z_0-\zeta)^j\\&
=&\displaystyle\sum_{j=0}^{w_{{j_l}}}\left(\sum_{k=1+v_{j_{l+1}}}^{+\infty}a_k{ k\choose j}\zeta^{k-j}\right) (z_0-\zeta)^j
\\&&\displaystyle+\ a_{1+w_{j_l}}(z_0-\zeta)^{1+w_{j_l}}+
\sum_{j=1+w_{j_l}}^m\left(\sum_{k=1+v_{j_{l+1}}}^{+\infty}a_k{ k\choose j}\zeta^{k-j}\right) (z_0-\zeta)^j\\
&:=&T_{1,l}(z_0)+\ a_{1+\mu_{n_{j_l}+1}}(z_0-\zeta)^{1+\mu_{{j_l}+1}}+ T_{2,l}(z_0).
\end{array}
$$
By (\ref{cons2_equ23}) we get $\vert a_{1+w_{j_l}}(z_0-\zeta)^{1+w_{j_l}}\vert\rightarrow 0,$ as $l$ tends to infinity. 
Again inspired by the estimates (\ref{LN_nouv_equ2}), (\ref{LN_nouv_equ3}) and 
(\ref{LN_equ3}) of the proof of Lemma \ref{lemme_npf} we get, for $i=1,2,$ $T_{i,l}(z_0)\rightarrow 0,$ as $l$ tends to infinity 
(let us recall that $m\in [1+\mu_{{j_l}+1},u_{j_{l+1}}]$ and $(1+v_{j_{l+1}})/u_{j_{l+1}}\rightarrow +\infty$). 
In summary we have, for all positive integer $m\in [w_{{j_l}},u_{j_{l+1}}]$, 
\begin{equation}\label{cons2_equ26}
S_{m}(f,\zeta)(z_0)\rightarrow 0,\hbox{ as }l\rightarrow +\infty. 
\end{equation}
By construction, we have 
\begin{equation}\label{cons2_equ25}\{\mu_n;n\in\mathbb{N}\}\bigcap
\left(\bigcup_{l\geq 1}([1+u_{j_{l+1}},w_{j_{l+1}}-1]\cap\mathbb{N})\right)=\emptyset.
\end{equation}
From (\ref{cons2_equ26}) and (\ref{cons2_equ25}), we get 
$$S_{\mu_n}(f,\zeta)(z_0)\rightarrow 0,\hbox{ as }n\rightarrow +\infty.$$
It follows that $f\notin \mathcal{U}^{(\mu)}(\mathbb{D},\zeta)$. This finishes the proof.
\end{proof}

\begin{example} {\rm Let $\zeta\in\mathbb{D}$, $\zeta\ne 0$. We can apply Theorem \ref{rec_thm} with the sequence $(\mu_n)=(n!)$. Therefore we obtain 
$\mathcal{U}^{((n!))}(\mathbb{D},0)\ne\mathcal{U}^{((n!))}(\mathbb{D},\zeta).$}
\end{example}

\begin{remarks}{\rm \begin{enumerate}
\item The combination of Corollary \ref{main_thm3} with Theorem \ref{rec_thm} gives Theorem \ref{main_thm_inde}.
\item Notice also that Theorem \ref{main_thm}, Corollary \ref{main_thm3} and Theorem \ref{rec_thm} 
remain valid for the classes of universal Taylor series 
$\mathcal{U}(\Omega,\zeta)$ and $\mathcal{U}^{(\mu)}(\Omega,\zeta)$ where you replace 
the unit disc $\mathbb{D}$ by a simply connected domain $\Omega,$ with $\Omega\subset\mathbb{C}$ and 
$\Omega\ne\mathbb{C}$, and $\zeta\in \Omega$, provided that the universal Taylor series 
possess Ostrowski-gaps (see \cite{MNadv, MVY}). To see this, it suffices to note that 
Lemma \ref{lemme_OG}, Lemma \ref{lemme_OG2}, Lemma \ref{lemme_npf} and Theorem \ref{thmCT} 
remain valid in this context. Thus the proofs work along the same lines.
\end{enumerate}
} 
\end{remarks}

\section{The real case} As far as we know the first example of universal series was introduced by Fekete \cite{Pal} 
which showed that there exists a real formal power series $\sum_{n\geq 1}a_n x^n$ satisfying the following universal property: for 
every continuous function $g$ on $[-1,1]$ with $g(0)=0$ there exists an increasing sequence $(\lambda_n)$ of positive integers such that 
$$\sup_{x\in [-1,1]}\left\vert \sum_{k=1}^{\lambda_n}a_kx^k-g(x)\right\vert\rightarrow 0,\hbox{ as }n\rightarrow +\infty.$$ 
Further combining this result with Borel's theorem we obtain $C^{\infty}$ functions vanishing at $0$ whose partial sums of its Taylor series with center $0$ 
approximate every continuous function vanishing at $0$ locally uniformly in $\mathbb{R}$ \cite{GE}. We denote by $C_0^{\infty}(\mathbb{R})$ the space 
of infinitely differentiable function on $\mathbb{R}$ vanishing at $0$.

\begin{definition} {\rm Let $\mu=(\mu_n)$ be an increasing sequence of positive integers with 
$\mu_n\rightarrow +\infty$ as $n$ tends to infinity. A function $f\in C_0^{\infty}(\mathbb{R})$ belongs to the class $\mathcal{U}^{(\mu)}$ of 
universal functions with respect to $\mu$ 
if for every compact set $K\subset\mathbb{R}$ and every continuous functions $h:\mathbb{R}\rightarrow\mathbb{R}$ with $h(0)=0$, there exists 
an increasing sequence $(\lambda_n)$ of positive integers such that
$$\sup_{x\in K}\left\vert S_{\lambda_n}(f)(x)-h(x)\right\vert\rightarrow 0,\hbox{ as }n\rightarrow +\infty.$$ 
For $\mu=\mathbb{N}$, we will denote $\mathcal{U}^{(\mu)}$ by $\mathcal{U}$.}
\end{definition}

In this context, the analogue of Theorem \ref{main_thm} states as follows. 

\begin{theorem}\label{main_thm_real} 
Let $\mu=(\mu_n)$ be a strictly increasing sequence of positive integers. The following assertions 
are equivalent:
\begin{enumerate} 
\item\label{1ass1} $\mathcal{U}=\mathcal{U}^{(\mu)}$
\item\label{2ass1} $\displaystyle\limsup_{n\rightarrow +\infty}\left(\frac{\mu_{n+1}}{\mu_n}\right)<+\infty.$
\end{enumerate}
\end{theorem}

To prove this result, we need the following results that are the $C^{\infty}$ versions of Lemma \ref{lemme_OG} and 
Theorem \ref{thmCT} respectively.

\begin{lemma}\label{lemme_OG_real}\cite[Proposition 4.4]{MouMun} 
Let $f\in\mathcal{U}$. Let $h:\mathbb{R}\rightarrow \mathbb{R}$ be a continuous 
function, with $h(0)=0.$ There exist two sequences of natural numbers $(\lambda_n),$ 
$(\mu_n)$ such that 
\begin{enumerate}
\item the Taylor series of $f$ around zero has Ostrowski-gaps 
$(\lambda_n,\mu_n),$
\item $S_{\lambda_n}(f)\rightarrow h,$ uniformly on each compact subset of 
$\mathbb{R}$ as $n\rightarrow +\infty.$
\end{enumerate}
\end{lemma}

\begin{lemma}\label{thmCT_real} \cite[Lemma 3.2]{Mou1} 
Let $(l_n)$ and $(m_n)$ be two strictly increasing sequences of positive integers such that 
$l_n\leq m_n$ and $\frac{m_n}{l_n}\rightarrow +\infty$ as $n\rightarrow +\infty.$ 
Let $A>0.$ For every continuous function $h:\mathbb{R}\rightarrow \mathbb{R},$ with $h(0)=0,$ there exists 
a sequence $(P_n)$ of real polynomials of the form $P_n(x)=\sum_{k=l_n}^{m_n}c_{n,k}x^k,$ such that
$$\sup_{x\in [-A,A]}\vert P_n(x)-h(x)\vert\rightarrow 0,\hbox{ as }n\rightarrow +\infty.$$
\end{lemma}

\vskip3mm

\noindent \textit{Sketch of the proof of Theorem \ref{main_thm_real}.} \\

$(\ref{2ass1})\Rightarrow(\ref{1ass1})$: thanks to Lemma \ref{lemme_OG_real}, it suffices to mimic the proof 
$(\ref{2ass})\Rightarrow(\ref{1ass})$ of Theorem \ref{main_thm2}.\\

$(\ref{1ass1})\Rightarrow(\ref{2ass1})$: we argue as in the proof of $(\ref{1ass})\Rightarrow(\ref{2ass})$ of Theorem \ref{main_thm2}. 
We define the same sequences $(u_j),$ $(v_j)$, $(w_j)$ and we denote by $(f_j)$ an enumeration of 
all the polynomials with coefficients in $\mathbb{Q}$. Using Lemma \ref{thmCT_real}, we construct step by step an increasing sequence of positive integers $(j_s)$ 
and polynomials 
and polynomials $P_i,\tilde{P}_i$, $i=1,\dots,s$ such that 
\begin{equation}\label{equaf11}
P_i(x)=\sum_{k=u_{j_i}}^{v_{j_i}}c_{j_i,k}x^k,\quad \tilde{P}_i(x)=\sum_{k=v_{j_i}+1}^{w_{j_i}}c_{j_i,k}x^k
\end{equation}
with
\begin{equation}\label{equaf21}\sup_{x\in [-s,s]}\left\vert \sum_{i=1}^{s-1}P_i(x) +P_{s}(x)-f_s(x)\right\vert <\frac{1}{(s+1)^2}
\end{equation}
\begin{equation}\label{equaf31}\sup_{x\in [-s,s]}\vert \tilde{P}_s(x)+f_s(x)\vert< \frac{1}{(s+1)^2}.
\end{equation}
We get by the triangle inequality 
\begin{equation}\label{equaf41}\begin{array}{rcl}\displaystyle\sup_{x\in [-s,s]}\left\vert \sum_{i=1}^{s}({P}_i(x)+ \tilde{P}_i(x))\right\vert&\leq &\displaystyle
\sup_{x\in [-s,s]}\left\vert \sum_{i=1}^{s}\left({P}_i(x)+\tilde{P}_i(x)\right)+{P}_{s}(x)-f_{s}(x)\right\vert\\&&\ 
\displaystyle + 
\sup_{x\in [-s,s]}\vert \tilde{P}_{s}(x)+f_{s}(x)\vert \\&<&\displaystyle\frac{2}{(s+1)^2}.\end{array}
\end{equation}
Finally let us consider the formal power series
$$\hat{f}(x)=\sum_{i\geq 1}\left(P_{i}(x)+\tilde{P}_i(x)\right).$$ 
By Borel's theorem one can find a function $f\in C_0^{\infty}(\mathbb{R})$ such that its Taylor development at zero is $\hat{f}$.  
The inequality $(\ref{equaf21})$ ensures that $f\in\mathcal{U}$. Moreover observe that the property (\ref{equa_suites_def}) 
implies that
$$\forall n\geq n_{j_1}+1,\forall s\geq 1\quad \mu_n\notin [u_{j_s},w_{j_s}-1]$$
and the equation (\ref{equaf41}) guarantees that, for all $n\in\mathbb{N},$
$$\sup_{x\in [-1,1]}\left\vert S_{\mu_n}(f)(x)\right\vert\leq 2\sum_{s\geq 1}\frac{1}{(s+1)^2}=\frac{\pi^2}{3}-2.$$
Thus $f\notin\mathcal{U}^{(\mu)}$. 
\hfill $\Box$\par\medskip 

\section*{Acknowledgments} The author was partly supported by the 
grant ANR-17-CE40-0021 of the French National Research Agency ANR (project Front).

\end{document}